\documentclass[12pt]{article}

\usepackage{hyperref}
\hypersetup{
    colorlinks=true,     
    linkcolor=blue,      
    citecolor=blue,      
    filecolor=blue,      
    urlcolor=blue        
}
\usepackage{amssymb}
\usepackage{amsmath}
\usepackage{amsthm}
\usepackage{authblk}
\usepackage{tikz}
\usepackage{tikz-cd}
\usepackage{cite}

\AtEndDocument{%
  \par
  \smallskip
  \begin{tabular}{@{}l@{}}%
    (Ethan Robinett) \textsc{Department of Mathematics,
    University of} \\ \textsc{California, Santa Barbara, CA 93106}\\
    \textit{E-mail address}: \texttt{robinett@math.ucsb.edu}
  \end{tabular}}
  \newtheorem*{con4}{Conjecture 1}
\newtheorem{con}{Conjecture}

\newtheorem{thm}{Theorem}
\newtheorem{prop}{Proposition}
\newtheorem{lemma}{Lemma}

\newtheorem*{con2}{Conjecture 2$'$}
\newtheorem*{con3}{Conjecture 3$'$}
\newtheorem*{ques}{Question}
\DeclareMathOperator{\sym}{Sym}
\DeclareMathOperator{\im}{im}
\DeclareMathOperator{\aut}{Aut}
\DeclareMathOperator{\id}{id}
\DeclareMathOperator{\gal}{Gal}
\DeclareMathOperator{\spn}{span}
\begin{document}

\title{Monoidal Networks}
\author{Ethan Robinett}
\date{}
\maketitle

\begin{abstract}
  \noindent In this paper we define and study the notion of a monoidal network, which consists of a commutative ring $R$ and a collection of groups $\Gamma_I$, indexed by the ideals of $R$, with $\Gamma_I$ acting on the quotient $R/I$ and satisfying a certain lifting condition. The examination of these objects is largely motivated by, and initially arose from, the study of the union-closed sets conjecture. This connection is made precise and other aspects of these structures are investigated. 
\end{abstract}

\noindent%

\section{Introduction}
\noindent In 1979, the following elementary problem was posed by Peter Frankl~\cite{Frankl}, known subsequently as the union-closed sets conjecture:
\begin{con4}\label{con1}
Let $S$ be a finite collection of finite sets with $|S| >1$. Suppose that the union of any two members of $S$ is also a member of $S$. Then there is some element that is contained in at least half of the sets of $S$. 
\end{con4}
\noindent Despite the incredibly simple nature of the given statement, this is a very difficult problem. To give the reader an idea of the current results in this direction, we give a brief account of the progress that has been made towards this conjecture (for a more comprehensive summary of this, the reader is referred to~\cite{Bruhn2}). The conjecture has been shown to hold when $|S| \leq 50$~\cite{Bruhn2} (this following from theorems contained in~\cite{Faro} and~\cite{zivkovic}), a bound which appears at the end of a line of results in this vein given by Sarvate and Renaud~\cite{Sarvate2, Sarvate}, Poonen~\cite{Poonen}, Faro~\cite{Faro}, Morris~\cite{Morris}, and Roberts~\cite{Roberts}. The claim is also known to be true if $m \doteq \Big| \displaystyle \bigcup_{A \in S} A \Big| \leq 12$~\cite{zivkovic}, and if $S$ contains either a singleton or a set with two elements~\cite{Bruhn2, Poonen, Sarvate2}. With regard to the latter result, it is also known that there is a major obstruction to generalizing the type of argument that leads to this conclusion, this is elaborated upon in~\cite{Bruhn2},~\cite{Poonen} and~\cite{Sarvate}. We also have results when $|S|$ is large in comparison to $m$, in particular, the claim is known to hold if $|S| \geq \lceil \frac{2}{3}2^m \rceil$~\cite{Balla}. This follows from the investigations of Balla, Bollob\'as and Eccles on the average size of member-sets of $S$ and represents an improvement of a previous result of this kind obtained by Nishimura and Takahashi~\cite{Nishimura}. All of the aforementioned statements place some additional hypothesis on $S$. Without any supplementary assumptions, as is noted in~\cite{Bruhn2}, it seems the strongest result available to us is due to Knill~\cite{Knill}, stating that there is always some element in at least $\frac{|S|-1}{\log_2 |S|}$ sets in $S$ (if $|S|$ is sufficiently large, this can be improved slightly to $\frac{2.4 |S|}{\log_2 |S|}$, see~\cite{wojcik}). For numerous other results, the reader is referred to each of the contributons cited above, as well as~\cite{Bruhn3, Zahar, Johnson, Norton, Roberts, Vaughan1, Vaughan2, Vaughan3}. 
\\
\indent \indent An aspect of this problem which at least partially motivates the content of this paper is perhaps best viewed by examining an alternative formulation of the conjecture. There is a lengthy list of lattice-theoretic statements given by Poonen in~\cite{Poonen}, each of which is equivalent to Conjecture~\ref{con1}. A good representative of the character of the items in this list is the sixth entry, which states that given a finite lattice $\mathcal{L}$ with $n= |\mathcal{L}| \geq 2$, there is some meet-irreducible $M \in \mathcal{L}$ such that at most $\frac{n}{2}$ elements $P \in \mathcal{L}$ satisfy $P \leq M$. This formulation of the problem has merit. It has allowed certain assumptions on the structure of $\mathcal{L}$ to be applied that would not have been obvious in the more set-theoretic setting, resulting in confirmation of the conjecture in additional special cases. In particular, the claim holds for distributive and geometric lattices~\cite{Poonen, Rival}, complemented lattices~\cite{Poonen}, modular lattices~\cite{Abe}, and more generally the so-called lower semimodular lattices~\cite{Reinhold}. 
\\
\indent \indent Other formulations can be found in~\cite{Bruhn1, Bruhn2, Zahar, Salzborn}. Like the above, each allows certain structural hypotheses to be brought to the problem, with additional results following. Moreover, each of these is similar to the lattice-theoretic version in the somewhat vague sense that the reliance is placed on the existence of a certain collection of objects, each of which is present in either at most or at least half of some other collection of objects, the specific type of object depending on the formulation. For instance, the graph-theoretic variant contained in~\cite{Bruhn2} states that any bipartite graph $G$ with at least one edge contains in each of its bipartition classes a vertex that lies in at most half of the maximal stable sets of $G$. Thus, here we are asking about the existence of vertices that appear in at most half of the maximal stable sets of a graph, and in the lattice-theoretic version we are asking about meet-irreducible elements of a lattice $\mathcal{L}$ appearing in the upper closures of at most half of the elements of $\mathcal{L}$. This general phenomenon is not surprising given that the original statement of the problem also fits this description, but the open status of Conjecture~\ref{con1} is evidence that conditions of this type are difficult to work with without additional conditions. 
\\
\indent \indent One purpose of what follows in this paper is to develop a formulation of Conjecture~\ref{con1} that does not rely upon the notion of some collection of distinguished objects appearing with a certain frequency. Specifically, we obtain an algebraic framework in which the conditions dictated by Conjecture~\ref{con1} can be studied in an abstract, systematic fashion. The inspiration for this stems from the simple observation that the conclusion of Conjecture~\ref{con1} does not hold if $S$ is not union-closed, suggesting that the conjecture may in fact be a deep structural statement disguised as a counting problem. Pursuing a formulation with the qualities mentioned above leads to the idea of a \emph{monoidal network}, as defined in Section~\ref{sec2} in terms of group actions on commutative rings. The relationship between these structures and Conjecture~\ref{con1} is given in detail in Section~\ref{sym}, where we also translate the supplementary conjectures found in~\cite{Poonen} into this language. While the precise definition of a monoidal network is rather technical and we therefore reserve it for the relevant section, it is worthy of note that in the context of Conjecture~\ref{con1}, a monoidal network can be viewed simply as a device which facilitates the use of mathematical induction. This aspect is intertwined with a certain generalization of the complement function on a power set called a \emph{pseudocomplement}, also defined in Section~\ref{sym}. Finally, the generality in which these objects are defined allows analogous questions to be asked in other algebraic situations. In fact, the various generalizations of Conjecture~\ref{con1} which become available to us are interesting in their own right. This is briefly explored in Section~\ref{auto}, which concludes with a concise discussion of existence-oriented questions.
\section{Definitions, conventions and notation} \label{sec2}
\noindent We begin by fixing some definitions and conventions.  All rings discussed henceforth are commutative and contain a multiplicative identity, and for a ring homomorphism $R \to R'$, we require that $1_R \mapsto 1_{R'}$. The notation $R^\times$ refers to the group of units in $R$, and $\aut (R)$ to the group of ring automorphisms of $R$. We write $R^{\aut (R)}$ for the subring of $R$ that is fixed pointwise by $\aut (R)$, henceforth referred to as the fixed subring of $R$. For nonempty $S \subseteq R$, the notation $\langle S \rangle$ refers to the ideal generated by $S$. If $S= \left\{ a \right\}$, we will write this simply as $(a)$. We emphasize that when we speak of a group $G$ acting on a ring $R$, we mean this in the purely set-theoretic sense, that is, we are viewing this as a homomorphism $G \to \sym (R)$ rather than (as is sometimes assumed) a homomorphism $G \to \aut (R)$. The ring structure of $R$ will be encoded in a different way. 
\\
\indent \indent We use the notation $(G, \gamma)$ to denote a group $G$ acting on a set $X$ with induced homomorphism $\gamma : G \to \sym (X)$. If $S \subseteq X$, we write $G^S$ for the set-wise stabilizer of $S$ under the action\footnote{We use this notation to avoid the subscript notation introduced in the widely-read book by Wielandt~\cite{Wielandt}, which means something else.}, evidently $G^S \leq G$. Given two group actions $(G, \gamma), (H, \phi)$ on $X$, we say the action of $H$ on $X$ \emph{extends} that of $G$ if $\im \gamma \subseteq \im \phi$. If $H$ is a subgroup of $G$, the notation $X/H$ refers to the orbit space of the action $(H, \gamma|_H)$. There is a natural projection $X \to X/H$, and we denote this map by $\theta$ (in virtually all situations in this paper, the subgroup $H$ associated with $\theta$ will be clear from context). 
\\
\indent \indent Now let $R$ be a ring. A \emph{multiaction} on $R$ is a collection $\Gamma (R) \doteq \left\{ (\Gamma_I, \gamma_I) \right\}_{I \trianglelefteq R}$, indexed by the ideals in $R$, with $(\Gamma_I, \gamma_I)$ an extension of the natural action of $\aut (R/I)$ on $R/I$. A generic multiaction on a ring $R$ will typically be denoted $\Gamma (R)$ as above, and we refer to the group $\Gamma_0$, corresponding to the zero ideal in $R$, as the \emph{top action group}. Following the convention established above, if $S \subseteq R/I$, the notation $\Gamma_I^S$ refers to the set-wise stabilizer of $S$ under the action of $\Gamma_I$ on $R/I$. One immediate example of a multiaction on $R$ is the \emph{automorphic} multiaction, given by setting $\Gamma_I \doteq \aut (R/I)$ for each ideal $I$ with the obvious action on $R/I$. This specific multiaction will be written as $\Gamma^{\aut} (R)$. Another example is the \emph{symmetric} multiaction, given by assigning $\Gamma_I \doteq \sym (R/I)$ with the natural set-theoretic action. We denote this by $\Gamma^{\sym} (R)$. One may also view each quotient $R/I$ as a module over $(R/I)^{\aut (R/I)}$, and then subsequently define $\Gamma_I$ to be the group of $(R/I)^{\aut (R/I)}$-module automorphisms of $R/I$. Each $\Gamma_I$ then extends the action of $\aut (R/I)$ in an obvious way, thus forming a multiaction on $R$. 
\\
\indent \indent We say a multiaction $\Gamma (R)$ is \emph{trivial} if there is a proper, nonzero ideal $I \subseteq R$ such that the action of $\Gamma_I$ on $R/I$ is trivial, and we say a ring $R$ is \emph{rigid} if $\aut (R)$ is the trivial group. We have the following elementary result:
\begin{prop} \label{prop:1} A ring $R$ admits a trivial multiaction if and only if $R/I$ is rigid for some proper, nonzero ideal $I$ of $R$.
 \begin{proof} If $R$ admits a trivial multiaction $\Gamma (R)$, then $R$ has a proper, nonzero ideal $I$ such that $(\Gamma_I, \gamma_I)$ is the trivial action on $R/I$. But $(\Gamma_I, \gamma_I)$ also extends the action of $\aut (R/I)$, so we have $\aut(R/I) \subseteq \im \gamma_I = \left\{ \id \right\}$, so $R/I$ is rigid. On the other hand, if $R/I$ is rigid for some proper, nonzero ideal $I$, then $\Gamma^{\aut} (R)$ is a trivial multiaction on $R$.
 \end{proof}
\end{prop} 

\indent \indent Thus far, we have defined a rather large collection of objects. Indeed, the defining properties of a generic multiaction are broad enough to allow for situations where the actions on the quotients $(\Gamma_I, \gamma_I)$ are not related to each other in any particular way. We now begin the process of defining a specific type of multiaction in which some of these actions are related to the action of the top action group on $R$. First, an \emph{extended submonoid} of $R$ is a multiplicative submonoid $M \subseteq R$  with $R^{\aut (R)} \subseteq M$. Now fix an extended submonoid $M$, a multiaction $\Gamma (R)$ and an ideal $I \subseteq R$. Let $\pi : R \to R/I$ be the canonical projection. We say $I$ is a \emph{covering ideal} of $M$ with associated submonoid $\hat{M}$ if there is an extended submonoid $\hat{M}$, contained in $M$ and with $M \setminus \hat{M} \subseteq I$, satisfying the following condition: for any $g \in \Gamma_I^{\pi (\hat{M})}$, there is a $\tilde{g} \in \Gamma_0^{\hat{M}}$ making the following commute:
\begin{center}
\begin{tikzcd}
\hat{M} \arrow[d, "\pi"] \arrow[rr, "\tilde{g}"] &  & \hat{M} \arrow[d, "\pi"] \\
\pi(\hat{M}) \arrow[rr, "g"] &  & \pi(\hat{M})
\end{tikzcd}
\end{center}
where, in the diagram above, by $g$ and $\tilde{g}$ we mean the bijection induced by these group elements, respectively. We say that $\tilde{g}$ is a \emph{lift} across $I$ of $g$. Finally, a \emph{monoidal network} on $R$ is a multiaction for which every extended submonoid has a proper, nonzero covering ideal. 
\\
\indent \indent A trivial multiaction on $R$ is always a monoidal network. In this case we have a proper, nonzero ideal $I$ such that $\Gamma_I$ acts trivially on $R/I$, thus $I$ is a covering ideal of any extended submonoid $M$ in $R$ and we may take $\hat{M} =M$, since the only bijection to be lifted across $I$ is the identity. On the other hand, no multiaction defined over a field is a monoidal network, as fields have no proper, nonzero ideals (the converse does not hold, that is, there are rings $R$ which are not fields but admit no monoidal network, see the conclusion of Section~\ref{auto}). Beyond these examples, it is difficult to discern whether a general multiaction on $R$ happens to be a monoidal network, though we give more examples at the conclusion of the next section in the case of $\Gamma^{\sym} (R)$ and most of Section~\ref{auto} is devoted to providing nontrivial examples in the case of $\Gamma^{\aut} (R)$. In fact, as will be shown in the next section, the assertion that $\Gamma^{\sym} (\mathbb{F}_2^n)$ is a monoidal network for every $n>1$ is already equivalent to the union-closed sets conjecture, so we should not expect easy answers to these questions, even when the structure of $R$ is well understood.

\section{The symmetric multiaction on $\mathbb{F}_2^n$} \label{sym}
\noindent This section is devoted to proving the following theorem:
\begin{thm} \label{thm:1}
The union-closed sets conjecture holds if and only if $\Gamma^{\sym} (\mathbb{F}_2^n)$ is a monoidal network for every $n>1$. 
\end{thm}
\noindent We work with the intersection-closed version of the conjecture, which states:
\begin{con}\label{con:1}
If $S \subseteq \mathcal{P}(X)$ is an intersection-closed set with $|S|>1$, then there is some $p \in X$ that is an element of at most half of the sets in $S$.
\end{con}
\noindent This is obtained immediately from the union-closed version by taking complements. By a result contained in~\cite{Poonen}, it is sufficient to consider only intersection-closed $S$ with $\emptyset, X_S \in S$, where $X_S \doteq \displaystyle \bigcup_{A \in S} A$. Following Poonen's notation in~\cite{Poonen}, for any $p \in X_S$, we fix the notation $S_p \doteq \left\{ A \in S : p \in A \right\}$. We recall from~\cite{Bruhn2} that a \emph{basis set} of an intersection-closed set $S$ is some $A \in S$ such that if $A= B \cap C$ for any $B,C \in S$, then either $B=A$ or $C=A$. We begin with the following easy lemma:
\begin{lemma} \label{lemma:1}
Let $S$ be intersection-closed with $|S| >1$ and $\emptyset, X_S \in S$. Then $p \in X_S$ has $|S_p| \leq \frac{|S|}{2}$ if and only if there is an intersection-closed set $\hat{S} \subseteq S$ with $\emptyset, X_S \in \hat{S}$ such that $|\hat{S}_p| = \frac{|\hat{S}|}{2}$ and for any $A \in S \setminus \hat{S}$, $p \notin A$. 
\begin{proof}
The backwards direction is immediate, if such an $\hat{S}$ exists, then we have $|S_p|=|\hat{S}_p|= \frac{|\hat{S}|}{2} \leq \frac{|S|}{2}$. We argue for the forward direction by induction on $|S| \geq 2$. If $|S|=2$, then we must have $S= \left\{ \emptyset,  X_S \right\}$. In this case, if $|S_p| \leq \frac{|S|}{2}$, then in fact $|S_p|=\frac{|S|}{2}$, thus we may take $\hat{S}=S$ and we are done. Now suppose the claim holds whenever $|\tilde{S}|=n$ for some $n \geq 2$ and let $|S|=n+1$. Suppose $p \in X_S$ has $|S_p| \leq \frac{|S|}{2}$. If $|S_p|= \frac{|S|}{2}$ then we are done, we may take $\hat{S}=S$ in this case, so assume $|S_p| <\frac{|S|}{2}$. Then $|S_p| \leq \frac{|S|-1}{2}$. Now since there is at least one $B \in S$ with $p \notin B$ and $|S|=n+1 \geq 3$, we may choose some basis set $A \in S$ with $p \notin A$, so that $\tilde{S}= S \setminus \left\{A \right\}$ is intersection-closed with $|\tilde{S}|=|S|-1=n$ and $\emptyset, X_S \in \tilde{S}$. Then note that $|\tilde{S}_p|=|S_p| \leq \frac{|S|-1}{2}= \frac{|\tilde{S}|}{2}$. By the inductive hypothesis, there is an intersection-closed $\hat{S} \subseteq \tilde{S} \subseteq S$ such that $|\hat{S}_p|= \frac{|\hat{S}|}{2}$ and for all $B \in \tilde{S}\setminus \hat{S}$, $p \notin B$. But $S \setminus \hat{S}= (\tilde{S} \setminus \hat{S}) \cup \left\{ A \right\}$, so since $p \notin A$, $\hat{S}$ is the desired intersection-closed set. This proves the claim.
\end{proof}
\end{lemma}
\indent \indent Lemma~\ref{lemma:1} is all that is needed to prove the forward direction of Theorem~\ref{thm:1}. We will actually prove something slightly stronger.
\begin{prop} \label{prop:2}
Suppose Conjecture~\ref{con:1} holds. Then for $n \geq 1$, any extended submonoid $M \subseteq \mathbb{F}_2^n$ has a maximal covering ideal under $\Gamma^{\sym} (\mathbb{F}_2^n)$. 
\begin{proof}
Fix $n \geq 1$ and choose some set $X$ with $|X|=n$. Endow $\mathcal{P}(X)$ with a ring structure by declaring $A+B \doteq A \Delta B$ and $A \cdot B \doteq A \cap B$. There is then an isomorphism $\Phi: \mathcal{P}(X) \xrightarrow{\cong} \mathbb{F}_2^n$ given by $\Phi: A \mapsto \chi_A$, where $\chi_A : X \to \mathbb{F}_2$ is the characteristic function of $A$. 
\\
\indent \indent Let $M \subseteq \mathbb{F}_2^n$ be an extended submonoid. The fixed subring of $\mathbb{F}_2^n$ is isomorphic to $\mathbb{F}_2$, thus $0,1 \in M$ and $\Phi^{-1}(M) \subseteq \mathcal{P}(X)$ is an intersection-closed set containing $\emptyset, X$. By Conjecture~\ref{con:1}, there is some $p \in X_{\Phi^{-1}(M)} \subseteq X$ with $|\Phi^{-1}(M)_p| \leq \frac{|\Phi^{-1}(M)|}{2}$, and Lemma~\ref{lemma:1} then gives an intersection-closed $\hat{S} \subseteq \Phi^{-1}(M)$ with $|\hat{S}_p|= \frac{|\hat{S}|}{2}$ and $p \notin A$ for all $A \in \Phi^{-1}(M)\setminus \hat{S}$. Set $\hat{M}= \Phi (\hat{S}) \subseteq M$ and $I= ( \Phi(X \setminus \left\{p\right\}))$. It is plain that $I$ is a maximal ideal of $\mathbb{F}_2^n$ with residue field $\mathbb{F}_2^n/I \cong \mathbb{F}_2$, we argue that under $\Gamma^{\sym} (\mathbb{F}_2^n)$,  $I$ is a covering ideal of $M$ with associated submonoid $\hat{M}$. 
\\
\indent \indent It is immediate that $\hat{M}$ is an extended submonoid in $\mathbb{F}_2^n$, since $\emptyset, X \in \hat{S}$. If we have $x \in M \setminus \hat{M}$, then $x= \chi_A$ for some $A \in \Phi^{-1}(M) \setminus \hat{S}$. Then $p \notin A$, hence $x \in I$, so that $M \setminus \hat{M} \subseteq I$. Now the restriction of the canonical projection $\pi: \mathbb{F}_2^n \to \mathbb{F}_2^n/I$ to $\hat{M}$ remains surjective, so we have $(\Gamma^{\sym}_I)^{\pi(\hat{M})}= \sym (\mathbb{F}_2)$. We need only guarantee that both $\id_{\mathbb{F}_2}$ and the transposition $(01)$ have lifts across $I$ to the stabilizer of $\hat{M}$ in the top action group $\Gamma_0^{\sym}= \sym (\mathbb{F}_2^n)$. For $\id_{\mathbb{F}_2}$ we may simply take $\id_{\mathbb{F}_2^n} \in (\Gamma_0^{\sym})^{\hat{M}}$, so consider $\phi= (01)$. Since $|\hat{S}_p|= \frac{|\hat{S}|}{2}$, we have $|I \cap \hat{M}|= \frac{|\hat{M}|}{2}$. We may then choose a bijection $\gamma: I \cap \hat{M} \to I^c \cap \hat{M}$, and we define $\tilde{\phi}: \mathbb{F}_2^n \to \mathbb{F}_2^n$ to be:
\begin{align*}
\tilde{\phi}: x \mapsto \begin{cases} 
      x & x \notin \hat{M} \\
      \gamma (x) & x \in I \cap \hat{M} \\
      \gamma^{-1}(x) & x \in I^c \cap \hat{M}. 
   \end{cases}
\end{align*}
Then $\tilde{\phi} \in (\Gamma_0^{\sym})^{\hat{M}}$ and we have $\phi \pi= \pi \tilde{\phi}$ on $\hat{M}$, as desired. The claim follows.
\end{proof}
\end{prop}
\noindent The forward direction of Theorem~\ref{thm:1} follows by noting that for $n>1$, any maximal ideal in $\mathbb{F}_2^n$ is proper and nonzero, thus $\Gamma^{\sym}(\mathbb{F}_2^n)$ is a monoidal network for $n>1$ if Conjecture~\ref{con:1} holds. 
\\
\indent \indent We move on to what is needed for the backwards direction of Theorem~\ref{thm:1}. First, consider the function $(\cdot )^c: \mathcal{P}(X) \to \mathcal{P}(X)$ sending $A \mapsto A^c$. This function has two important properties that we wish to distill. The first is that it is an involution, and the second is that the ideal in $\mathcal{P}(X)$ generated by the elements of the form $A \cdot A^c$ is proper (indeed, it is the zero ideal). We then make the following definition: if $M \subseteq \mathbb{F}_2^n$ is an extended submonoid and $T: M \to M$ is an involution, we define the \emph{quadratic ideal} of $T$ to be $\Lambda (T) \doteq \langle xT(x) \rangle_{x \in M}$, and we say $T$ is a \emph{pseudocomplement} if the quadratic ideal of $T$ is proper. We have the following equivalent of Conjecture~\ref{con:1}:
\begin{prop} \label{prop:3}
Conjecture~\ref{con:1} holds if and only if for any $n \geq 1$ and extended submonoid $M \subseteq \mathbb{F}_2^n$, there is a pseudocomplement $M \to M$. 
\begin{proof} Suppose Conjecture~\ref{con:1} holds and let $M\subseteq \mathbb{F}_2^n$ be an extended submonoid. By Proposition~\ref{prop:2}, $M$ has a maximal covering ideal, say $I$, with associated submonoid $\hat{M} \subseteq M$. Let $\phi=(01) \in (\Gamma_I^{\sym})^{\pi(\hat{M})}$ and let $\tilde{\phi} \in (\Gamma_0^{\sym})^{\hat{M}}$ be the lift across $I$ of $\phi$ constructed at the end of the proof of Proposition~\ref{prop:2}. Notice that $\tilde{\phi} \in (\Gamma_0^{\sym})^{M}$, and in fact $\tilde{\phi}$ is an involution by construction. We now show that $\Lambda(\tilde{\phi}|_M)$ is a proper ideal of $\mathbb{F}_2^n$.  We argue that $\pi (\Lambda(\tilde{\phi}|_M)) =0$, this implies $\Lambda(\tilde{\phi}|_M) \subseteq I$, which is sufficient. Now, if $x \in M \setminus \hat{M}$, then $x \in I$ and we also have $x \tilde{\phi}(x)=x^2 \in I$, thus $\pi (x \tilde{\phi}(x))=0$, so suppose $x \in \hat{M}$. We have:
\begin{align*}
\pi(x \tilde{\phi}(x)) &= \pi(x) (\pi \circ \tilde{\phi})(x) \\
&= \pi(x) (\phi \circ \pi)(x) \\
&= \pi (x) \phi (\pi (x)) =0
\end{align*}
where the second line follows since $\tilde{\phi}$ lifts $\phi$ on $\hat{M}$ and the third follows since $\mathbb{F}_2^n/I \cong \mathbb{F}_2$ and $\phi=(01)$. The above shows that the restriction $\tilde{\phi}|_M$ is a pseudocomplement on $M$, this proves the forward direction. 
\\
\indent \indent For the backwards direction, let $\Phi: \mathcal{P}(X) \to \mathbb{F}_2^n$ be the isomorphism from Proposition~\ref{prop:2} and choose some intersection-closed $S \subseteq \mathcal{P}(X)$ with $\emptyset, X \in S$. Then $M \doteq \Phi(S)$ is an extended submonoid, by hypothesis there is a pseudocomplement $T: M \to M$. The map $T_* \doteq \Phi^{-1} \circ T \circ \Phi$ is an involution on $S$, we note that $P \doteq \left\{ \left\{A, T_* (A) \right\} \right\}_{A \in S}$ forms a partition of $S$. Now since $\Lambda (T)$ is a proper ideal of $\mathbb{F}_2^n$, so is the ideal $\langle A \cap T_*(A) \rangle_{A \in S}$ in $\mathcal{P}(X)$. In particular, there is a $p \in X$ such that for each $A \in S$, $A \cap T_* (A) \subseteq X \setminus \left\{ p \right\}$. Then for any $A \in S$, either $p \notin A$ or $p \notin T_*(A)$, it follows that for each $C \in P$, $|S_p \cap C| \leq \frac{|C|}{2}$. Immediately, then:
\begin{align*}
|S_p|= \displaystyle \sum_{C \in P} |S_p \cap C| 
\leq \frac{1}{2}\displaystyle \sum_{C \in P} |C| 
= \frac{|S|}{2}.
\end{align*}
This proves the claim.
\end{proof}
\end{prop}
\indent \indent Proposition~\ref{prop:3} gives a formulation of Conjecture~\ref{con:1} without any mention of $S_p$ or its cardinality. We pause momentarily to make an observation about what Conjecture~\ref{con:1} is actually saying in view of this: we already have a \emph{global} pseudocomplement $\mathbb{F}_2^n \to \mathbb{F}_2^n$, given by $x \mapsto x+1$. The question is then whether one of these always exists at the more localized level of extended submonoids. 
\\
\indent \indent In order to prove the backwards direction of Theorem~\ref{thm:1}, it suffices to show that if $\Gamma^{\sym}(\mathbb{F}_2^n)$ is a monoidal network for every $n>1$, then every extended submonoid in $\mathbb{F}_2^n$ has a pseudocomplement. The class of rings of the form $\mathbb{F}_2^n$ has the property that it is closed (up to isomorphism) under quotients by proper ideals, we wish to exploit this to prove the claim. As it happens, the hypothesis that $\Gamma^{\sym}(\mathbb{F}_2^n)$ forms a monoidal network for each $n>1$ is precisely what is needed for the inductive lifting of pseudocomplements.
\begin{prop}\label{prop:4}
Suppose $\Gamma^{\sym}(\mathbb{F}_2^n)$ is a monoidal network for every $n>1$. Then any extended submonoid $M \subseteq \mathbb{F}_2^n$, $n \geq 1$, has a pseudocomplement.
\begin{proof}
We argue by induction on $n$. For $n=1$, the only extended submonoid of $\mathbb{F}_2$ is $\mathbb{F}_2$ itself, for which the transposition $(01)$ is a pseudocomplement. Now suppose the claim holds for all $1 \leq m<n$ and let $M \subseteq \mathbb{F}_2^n$ be an extended submonoid. By hypothesis, $\Gamma^{\sym}(\mathbb{F}_2^n)$ is a monoidal network, therefore $M$ has a proper, nonzero covering ideal $I$ with some associated submonoid $\hat{M} \subseteq M$. Now we have $\mathbb{F}_2^n/I \cong \mathbb{F}_2^m$ for some $1 \leq m<n$, since $I$ is proper and nonzero. The fixed subring of $\mathbb{F}_2^n/I$ is then isomorphic to $\mathbb{F}_2$ and is contained in the multiplicative submonoid $\pi (\hat{M})$, it follows that $\pi (\hat{M})$ is an extended submonoid of $\mathbb{F}_2^n/I$. By the inductive hypothesis, there is a pseudocomplement $T: \pi (\hat{M}) \to \pi(\hat{M})$. Pseudocomplements are bijections, so we may assume that $T \in \sym (\mathbb{F}_2^n/I)$, after perhaps extending $T$ to the whole quotient. Then $T \in (\Gamma_I^{\sym})^{\pi (\hat{M})}$, so there is a lift of $T$ along $I$ to some $\tilde{T} \in (\Gamma_0^{\sym})^{\hat{M}}$. We are not guaranteed that the restriction of $\tilde{T}$ to $M$ is a pseudocomplement on $M$. We can, however, construct one from $\tilde{T}$. We begin this process now.
\\
\indent \indent For each $a \in \pi (\hat{M})$, set $V_a \doteq \left\{ \pi^{-1}(a) \cap \hat{M}, \pi^{-1}(T(a)) \cap \hat{M} \right\}$ and note that since $T^2= \id$ and $T\pi = \pi \tilde{T}$ on $\hat{M}$, $\tilde{T}$ gives a bijection $\pi^{-1}(a) \cap \hat{M} \to \pi^{-1}(T(a)) \cap \hat{M}$. Let $\psi : \left\{ V_a \right\}_{a \in \pi (\hat{M})} \to \displaystyle \bigcup_{ a \in \pi (\hat{M})} V_a$ have $\psi (V_a) \in V_a$ for each $a \in \pi (\hat{M})$. Define $\gamma : M \to M$ by:
\begin{align*}
\gamma: x \mapsto \begin{cases} 
      x & x \in M \setminus \hat{M} \\
      \tilde{T}(x) & x \in \psi (V_a) \\
      \tilde{T}^{-1}(x) & x \in B \in V_a, B \neq \psi (V_a). 
   \end{cases}
\end{align*}
The function given above is well-defined since the collection $\displaystyle \bigcup_{ a \in \pi (\hat{M})} V_a$ is a partition of $\hat{M}$. It is also an involution on $M$, since $\tilde{T}$ transposes the elements of $V_a$. To complete the proof, we need to show that the quadratic ideal of $\gamma$ is proper. We will do this by showing that $\pi (\Lambda (\gamma)) \subseteq \Lambda (T)$.  Let $x \in M$. If $x \in M \setminus \hat{M}$, then $x\gamma (x)=x^2$. Now $M \setminus \hat{M} \subseteq I$, so  that $x^2 \in I$ and $\pi (x \gamma (x))=0 \in \Lambda (T)$. Now suppose $x \in \hat{M}$. Then $x \in B \in V_a$ for some $a \in \pi(\hat{M})$. We always have, in this case, $x \gamma (x)=y \tilde{T}(y)$ for some $y \in \psi (V_a)$. Indeed, if $x \in \psi (V_a)$ this follows immediately. Otherwise, $B \neq \psi (V_a)$ and $\tilde{T}$ gives a bijection $\psi (V_a) \to B$, so $x= \tilde{T}(y)$ for some $y \in \psi (V_a)$ and $x \gamma (x)= y \tilde{T}(y)$. Note then that:
\begin{align*}
\pi( x \gamma (x))= \pi(y) \pi( \tilde{T}(y))= \pi (y) T(\pi(y)) \in \Lambda (T)
\end{align*}
since $\pi \tilde{T}= T \pi$ on $\hat{M}$. This shows that $\pi (x \gamma (x)) \in \Lambda (T)$ for every $x \in M$, it follows that $\gamma$ is a pseudocomplement on $M$.
\end{proof}
\end{prop}
\indent \indent As was remarked earlier, the forward direction of Theorem~\ref{thm:1} follows from Proposition~\ref{prop:2} and now the backwards direction follows from Propositions~\ref{prop:3} and~\ref{prop:4}, thus the proof is complete. We are also now in a position to answer a question that the inquisitive reader may have regarding all that has been said up to this point: why do we use extended submonoids rather than simply multiplicative submonoids of $\mathbb{F}_2^n$? It is easy to see that the entire machinery developed above breaks down if we do not require $\hat{M}$ and $M$ to be extended submonoids, for otherwise we may simply choose a maximal ideal $I$ of $\mathbb{F}_2^n$, form the multiplicative submonoid $I^c \cap M$, and note that its image under projection is a singleton. Trivially, then, $I$ would be a ``covering ideal" of $M$, but this would say nothing at all about $M$ itself.
\\
\indent \indent It is natural at this time to ask similar questions for $\Gamma^{\sym} (R)$ for commutative rings $R$ other than $\mathbb{F}_2^n$, i.e. when exactly is this a monoidal network? It is easy to make the answer go either way. On the one hand, the naive generalization of Conjecture~\ref{con:1} fails when we replace $\mathbb{F}_2$ with $\mathbb{F}_p$ for prime $p>2$:
\begin{prop} \label{prop:5}
If $p>2$ is prime, then $\Gamma^{\sym} (\mathbb{F}_p^n)$ is not a monoidal network for any $n \geq 1$. 
\begin{proof}
Suppose $\Gamma^{\sym} (\mathbb{F}_p^n)$ is a monoidal network, we will prove that $p=2$. If $n=1$ then $\Gamma^{\sym} (\mathbb{F}_p^n)$ is not a monoidal network since $\mathbb{F}_p$ is a field, so $n>1$. Consider the extended submonoid $M \doteq (\mathbb{F}_p^n)^\times \cup \left\{0 \right\}$. By hypothesis, $M$ has a proper, nonzero covering ideal $I$ with some associated submonoid $\hat{M} \subseteq M$. Now $0 \in \hat{M}$, so any $x \in M \setminus \hat{M}$ is a unit, but $M \setminus \hat{M} \subseteq I$, so since $I$ is proper we must have $M= \hat{M}$. We have $\mathbb{F}_p^n/I \cong \mathbb{F}_p^m$ for some $1 \leq m<n$, and we also have $\pi (M)= (\mathbb{F}_p^m)^\times \cup \left\{0 \right\}$. By definition, for each $ \phi \in (\Gamma_I^{\sym})^{\pi (M)}= \sym (\pi (M))$ there is a $\tilde{\phi} \in (\Gamma_0^{\sym})^M= \sym (M)$ with $\phi \pi=\pi \tilde{\phi}$ on $M$. Since $(\Gamma_I^{\sym})^{\pi (M)}$ acts transitively on $\pi (M)$, it follows that for any two $a,b \in \pi (M)$, $|\pi^{-1}(a) \cap M|= |\pi^{-1}(b) \cap M|$. But $|\pi^{-1}(0) \cap M|=1$, so:
\begin{align*}
(p-1)^n+1= |M|= \displaystyle \sum_{a \in \pi (M)} |\pi^{-1}(a) \cap M| = \displaystyle \sum_{a \in \pi (M)} 1= |\pi (M)|= (p-1)^m+1.
\end{align*}
Since $m<n$, we conclude that $p=2$.
\end{proof}
\end{prop}
\noindent In the other direction, there is a broad class of rings for which $\Gamma^{\sym}(R)$ is a monoidal network.
\begin{prop}\label{prop:6}
If $R$ is rigid and not a field, then $\Gamma^{\sym} (R)$ is a monoidal network.
\begin{proof}
Since $R$ is rigid, $R^{\aut (R)}=R$, so the only extended submonoid of $R$ is $R$ itself. We then need only guarantee that $R$ has a proper, nonzero covering ideal under $\Gamma^{\sym} (R)$. For this, let $I \subseteq R$ be any proper, nonzero ideal. Given $a,b \in R/I$, let $\gamma_{a,b}: \pi^{-1}(a) \to \pi^{-1}(b)$ be a bijection. Then given $ \phi \in \Gamma_I^{\sym}= \sym (R/I)$, the map:
\begin{align*}
\tilde{\phi}: (x \in \pi^{-1}(a)) \mapsto \gamma_{a, \phi(a)} (x)
\end{align*}
is easily checked to be a lift along $I$ of $\phi$.
\end{proof}
\end{prop}
\noindent So, for instance $\Gamma^{\sym} (\mathbb{Z})$ is a monoidal network. The proof of Proposition~\ref{prop:6} exploits the fact that $R^{\aut (R)}$ is ``large" in this case, and one may hope to construct further examples of $R$ for which $\Gamma^{\sym} (R)$ is a monoidal network by finding rings with ``large" fixed subrings. It is interesting to note that Conjecture~\ref{con:1} asserts something that is essentially the opposite of this intuition: the fixed subring of $\mathbb{F}_2^n$ is always very small, isomorphic to $\mathbb{F}_2$, but Conjecture~\ref{con:1} asserts that $\Gamma^{\sym} (\mathbb{F}_2^n)$ is a monoidal network for $n>1$ regardless.
\\
\indent \indent We conclude this section by restating Poonen's conjectures~\cite{Poonen} in the language that has been introduced thus far. As this is mostly an exercise in translation, we omit proofs, with one exception that we will discuss. The first of these (Conjecture 1 in~\cite{Poonen}) is just Conjecture~\ref{con:1} here, the restatement of this is the content of Theorem~\ref{thm:1}. The statements of the remaining conjectures for intersection-closed sets are listed below\footnote{The original statements rely on the idea of a \emph{block}, which is an equivalence class of elements of $X_S$, two elements being equivalent if they lie in exactly the same sets in $S$. The intersection-closed sets with singleton blocks are commonly called \emph{separating}~\cite{Bruhn2}.}:
\begin{con} \label{con:2}
Let $S$ be a separating intersection-closed set. Then either $S$ is a power set or there is some $\alpha \in X_S$ with $|S_{\alpha}|<\frac{|S|}{2}$. 
\end{con}
\begin{con} \label{con:3}
Let $S$ be an intersection-closed set. If there is exactly one $\alpha \in X_S$ with $|S_{\alpha}| \leq \frac{|S|}{2}$, then $X_S$ is the only set in $S$ containing $\alpha$. 
\end{con}
\begin{con} \label{con:4}
Let $S$ be a separating intersection-closed set with $|S|>1$. If there is exactly one $\alpha \in X_S$ with $|S_{\alpha}| \leq \frac{|S|}{2}$, then $S$ is precisely the collection of all subsets of $X_S$ which don't contain $\alpha$, together with $X_S$ itself. 
\end{con}
\noindent We will restate Conjectures~\ref{con:2} and ~\ref{con:3}, then prove that Conjecture~\ref{con:4} holds if and only if Conjectures~\ref{con:2} and \ref{con:3} do. This is Proposition 1 in~\cite{Poonen}, where it is proven under the additional assumption that Conjecture~\ref{con:1} holds, we thus remove that assumption. 
\\
\indent \indent We can remove the separating hypothesis in Conjecture~\ref{con:2} by replacing ``power set" with ``$\sigma$-algebra" (more specifically, the remark preceding Conjecture~\ref{con:2} in~\cite{Poonen} concerning the replacement of elements with blocks in power sets is a construction that yields precisely the $\sigma$-algebras contained in a given (finite) power set. More formally, every $\sigma$-algebra on a finite set is generated by a partition, see~\cite{Tao}, for instance).  The advantage to this is twofold. Firstly, $\sigma$-algebras correspond exactly to subrings of $\mathbb{F}_2^n$. Secondly, the separating hypothesis seems cumbersome to translate, if we wish to present this in a form in which it could be considered for other rings, it is in our best interest to remove it. The new version of Conjecture~\ref{con:2} reads:
\begin{con2}
If $M \subseteq \mathbb{F}_2^n$ is an extended submonoid that is not a subring of $\mathbb{F}_2^n$, then $M$ has a maximal covering ideal with an associated submonoid that is properly contained in $M$. 
\end{con2}
\indent \indent Conjecture~\ref{con:3}, when restated, is a statement about the relationship between uniqueness of maximal covering ideals and uniqueness of associated submonoids. Fixing an extended submonoid $M \subseteq \mathbb{F}_2^n$, say an extended submonoid $M'$ contained in $M$ is \emph{covered} if it is associated to some proper, nonzero covering ideal of $M$. The new version of Conjecture~\ref{con:3} can then be stated as:
\begin{con3}
If $M \subseteq \mathbb{F}_2^n$ is an extended submonoid with a unique maximal covering ideal, then $\mathbb{F}_2$ is the only covered submonoid contained in $M$. 
\end{con3}
\noindent Finally, for Conjecture~\ref{con:4} we have the following result:
\begin{prop} \label{prop:8} 
Conjecture~\ref{con:4} holds if and only if Conjectures~\ref{con:2} and~\ref{con:3} both do.
\begin{proof}
It is known that this holds under the assumption of Conjecture~\ref{con:1}. Moreover, the implications $(4) \Rightarrow (3)$ and $(2) \Rightarrow (1)$ hold unconditionally. In particular, the backwards direction of the claim holds unconditionally, so it suffices to prove the forward direction. For this, it is enough to prove the implication $(3) \Rightarrow (1)$. Indeed, if this holds, then $(4) \Rightarrow (3) \Rightarrow (1)$ holds unconditionally, thus the forward direction also holds unconditionally since it holds under the assumption of $(1)$.
\\
\indent \indent Supposing $(3)$ holds, let $S$ be intersection-closed with $\emptyset, X_S \in S$ and $X_S \neq \emptyset$. Pick any $\alpha \notin X_S$ and consider the intersection-closed set $\tilde{S}= S \cup \left\{ \left\{\alpha \right\}, X_S \cup \left\{ \alpha \right\} \right\}$. Immediately, then $|\tilde{S}_{\alpha}| \leq \frac{|\tilde{S}|}{2}$. If $\alpha$ were the only such element of $X_{\tilde{S}}= X_S \cup \left\{ \alpha \right\}$, then Conjecture~\ref{con:3} would demand that $X_{\tilde{S}}$ be the only set in $\tilde{S}$ containing $\alpha$, but it certainly is not. Then there is some $\beta \neq \alpha$, $\beta \in X_{\tilde{S}}$ with $|\tilde{S}_{\beta}| \leq \frac{|\tilde{S}|}{2}$. It follows that, since $\beta$ is contained in exactly one member-set of $\tilde{S}\setminus S$, we have $|S_{\beta}| \leq \frac{|S|}{2}$. Then $(1)$ holds and we are finished.
\end{proof}
\end{prop}
\section{The automorphic multiaction} \label{auto}
\noindent The purpose of this section is to provide a context in which monoidal networks occasionally occur that is not directly related to Conjecture~\ref{con:1}, in turn motivating the abstract definition of these objects. More specifically, we briefly investigate $\Gamma^{\aut} (R)$, thereby producing examples in which this multiaction forms a monoidal network. This is quite important, as the condition that $\Gamma^{\aut}(R)$ form a monoidal network is very stringent and it is not immediately clear that nontrivial examples of this actually exist. We first recall that a ring of prime characteristic is called \emph{perfect} (in the sense of Serre~\cite{Serre}) if the Frobenius endomorphism on $R$ is an automorphism. The first such example comes easily when $R$ is a perfect ring with a certain property.
\begin{prop} \label{prop:7}
Let $R$ be a perfect ring of characteristic $p>0$. Suppose $R$ is not a field and additionally that the Frobenius on $R$ has finite order. Then $\Gamma^{\aut} (R)$ is a monoidal network.
\end{prop}
\begin{proof}
Let $\gamma: x \mapsto x^p$ be the Frobenius on $R$. Our first observation is that the finite order of $\gamma$ implies that any residue field of $R$ is necessarily finite. Indeed, let $I \subseteq R$ be maximal and let $k>0$ be the order of $\gamma$. If $\phi : R/I \to R/I$ is the Frobenius endomorphism on $R/I$, then the following diagram commutes:
\begin{center}
\begin{tikzcd}
R \arrow[d, "\pi"] \arrow[rr, "\gamma^r"] &  & R \arrow[d, "\pi"] \\
R/I \arrow[rr, "\phi^r"] &  & R/I
\end{tikzcd}
\end{center}
for any $r>0$. In particular, setting $r=k$ gives that $x^{p^k}=x$ for any $x \in R/I$, it follows that $R/I$ is finite.
\\
\indent \indent Now fix a maximal ideal $I \subseteq R$. Since $R$ is not a field, $I$ is proper and nonzero and by the above we may write $R/I \cong \mathbb{F}_{p^n}$ for some $n \geq 1$. Let $M \subseteq R$ be an extended submonoid and let $t \in (\Gamma_I^{\aut})^{\pi (M)}$. Then $t \in \aut (R/I)= \gal (\mathbb{F}_{p^n}/\mathbb{F}_p)$, so in fact $t= \phi^m$ for some $m \geq 0$. Then $\pi \gamma^m=t\pi$ on $M$, so it suffices to show that $\gamma^m \in (\Gamma_0^{\aut})^M$, but this is immediate since $\gamma (M) \subseteq M$ and $\gamma^{-1}=\gamma^{k-1}$. 
\end{proof}
\indent \indent Explicitly, then, take $S$ to be some set with $|S| \geq 2$ and for each $i \in S$, let $F_i$ be a finite extension of $\mathbb{F}_p$. If there is some $k>0$ such that $[F_i : \mathbb{F}_p] \leq k$ for every $i$, then the automorphic multiaction on the product $\displaystyle \prod_{i \in S} F_i$ is a monoidal network by Proposition~\ref{prop:7}. Another class of examples is furnished by the backwards direction of Proposition~\ref{prop:1}: if $R/I$ is rigid for a proper, nonzero ideal $I$, then $\Gamma^{\aut}(R)$ is a (trivial) monoidal network. We thus seek examples not of the types mentioned above, which were relatively easy to find. A convenient setting for this is the class of finite local rings, for reasons which will soon become clear. Upon examining rings of the form $\mathbb{F}_{p^n}[x]/(x^2)$, we obtain the following theorem:
\begin{thm} \label{thm:2}
For any prime $p$, $\Gamma^{\aut} (\mathbb{F}_{p^n}[x]/(x^2))$ is a monoidal network if and only if $n \leq 3$. 
\end{thm}
\indent \indent Theorem~\ref{thm:2} makes precise the remarkable fact that deciding whether the automorphic multiaction on $\mathbb{F}_{p^n}[x]/(x^2)$ is a monoidal network is purely a question of determining the residue field extension degree $n$. This will require some effort to prove, the case $n=4$ being noticeably harder than the others. Before proceeding, we remark that Theorem~\ref{thm:2} does give examples of $R$ for which $\Gamma^{\aut} (R)$ is a monoidal network that are not of the previously mentioned types. These rings are certainly not perfect, and for $n>1$ the corresponding multiactions are nontrivial. 
\\
\indent \indent There is a general situation that will need to be dealt with repeatedly in order to prove Theorem~\ref{thm:2}, we therefore establish a piece of terminology for this. Suppose $(G, \gamma)$ is an action on some set $X$, and let $N,H$ be subgroups of $G$ with $N$ normal in $G$. There is then a unique induced action of $H$ on $X/N$ which completes the following diagram for any $h \in H$:
\begin{center}
\begin{tikzcd}
X \arrow[rr, "h"] \arrow[d, "\theta"] &  & X \arrow[d, "\theta"] \\
X/N \arrow[rr, "h", dashed]           &  & X/N                  
\end{tikzcd}
\end{center}
It is stressed that this action is well-defined only because $N$ is normal in $G$. We say $N$ \emph{trivializes} $H$, or is a \emph{trivialization} of $H$, under $(G, \gamma)$ if this induced action is trivial. Another way of saying this is that, under the restriction of $(G, \gamma)$ to $H$, we have $H^{\theta (x)} = H$ for any $x \in X$. Thus, if $N$ trivializes $H$, then the action of $H$ on $X$ is ``broken up" into separate actions on each member-set of the partition $\left\{ \theta (x): x \in X \right\}$. 
\\
\indent \indent Moving towards a proof of Theorem~\ref{thm:2}, henceforth we let $R_{n,p}\doteq \mathbb{F}_{p^n}[x]/(x^2)$, for brevity, and we let $\varepsilon$ be the image of $x \in \mathbb{F}_{p^n}[x]$ in $R_{n,p}$ under projection. Then the set $\left\{ 1, \varepsilon \right\}$ is an $\mathbb{F}_{p^n}$-basis for $R_{n,p}$. We define, for $\lambda \in \mathbb{F}_{p^n}^\times$, $T_\lambda : a+b\varepsilon \mapsto a+b \lambda \varepsilon$, this is a ring automorphism of $R_{n,p}$. Similarly, we let $\sigma: a+b\varepsilon \mapsto a^p+b^p \varepsilon$, which is also a ring automorphism on $R_{n,p}$. We identify the subgroup of $\aut (R_{n,p})$ consisting of the $T_\lambda$ with $\mathbb{F}_{p^n}^\times$ and the subgroup generated by $\sigma$ with $\gal (\mathbb{F}_{p^n}/\mathbb{F}_p)$. Standard arguments then give that $\aut (R_{n,p})= \mathbb{F}_{p^n}^\times \rtimes \gal (\mathbb{F}_{p^n}/\mathbb{F}_p)$ (i.e. $\aut (R_{n,p})$ is the so-called general semilinear group in one dimension over $\mathbb{F}_{p^n}$, as defined in~\cite{Gruenberg}), and that $R_{n,p}^{\aut (R_{n,p})}= \mathbb{F}_p$. We have the following lemma:
\begin{lemma} \label{lemma:2}
For $0 \leq r \leq n$, let $S_r = \left\{ H \leq R_{n,p}^\times : \mathbb{F}_{p^n}^\times \leq H, |H|=p^r (p^n-1) \right\}$. There is a natural action of $\aut (R_{n,p})$ on $S_r$, and $\Gamma^{\aut} (R_{n,p})$ is a monoidal network if and only if $\mathbb{F}_{p^n}^\times$ trivializes $\gal (\mathbb{F}_{p^n}/\mathbb{F}_p)$ under this action for each $r$. 
\begin{proof}
We argue for the forward direction first. Suppose $\Gamma^{\aut} (R_{n,p})$ is a monoidal network and let $H \in S_r$. Then we have $\mathbb{F}_p \subseteq H \cup \left\{0 \right\} \doteq H_0$, so $H_0$ is an extended submonoid of $R_{n,p}$. The only proper, nonzero ideal of $R_{n,p}$ is $I= (\varepsilon )$, so this must be a covering ideal of $H_0$. Moreover, any nonzero element of $H_0$ is necessarily a unit, so the condition $H_0 \setminus \hat{H_0} \subseteq I$ implies the extended submonoid associated to $I$ is $\hat{H_0}=H_0$. Now, by definition of $S_r$, $\mathbb{F}_{p^n} \subseteq H_0$, so that $\pi (H_0)= R_{n,p}/I \cong \mathbb{F}_{p^n}$. Letting $\phi$ be the Frobenius on $R_{n,p}/I$, we have $\phi \in (\Gamma_I^{\aut} )^{\pi (H_0)}$, so there is some $\tilde{\phi} \in (\Gamma_0^{\aut})^{H_0}$ making the following commute:
\begin{center}
\begin{tikzcd}
H_0 \arrow[d, "\pi"] \arrow[rr, "\tilde{\phi}"] &  & H_0 \arrow[d, "\pi"] \\
\mathbb{F}_{p^n} \arrow[rr, "\phi"] &  & \mathbb{F}_{p^n}
\end{tikzcd}
\end{center}
The decomposition $\aut (R_{n,p})= \mathbb{F}_{p^n}^\times \rtimes \gal (\mathbb{F}_{p^n}/\mathbb{F}_p)$ allows us to write $\tilde{\phi}= T_\lambda \sigma^k$ for some $0 \leq k < n$ and $\lambda \in \mathbb{F}_{p^n}^\times$, and the above diagram, together with the fact that $\mathbb{F}_{p^n} \subseteq H_0$, implies that $k=1$, thus $\tilde{\phi}=T_\lambda \sigma$. In particular, we have $T_\lambda \sigma \in \aut (R_{n,p})^{H_0}$ and since this automorphism fixes $0$, actually we have $T_\lambda \sigma \in \aut (R_{n,p})^H$. Then $\sigma (H)= T_{\lambda^{-1}}(H)$, and it follows that if $\theta : S_r \to S_r/ \mathbb{F}_{p^n}^{\times}$ is the projection, $\theta (\sigma (H))= \theta (H)$. Since the copy of $\gal ( \mathbb{F}_{p^n}/ \mathbb{F}_p)$ in $\aut (R_{n,p})$ is generated by $\sigma$, it follows that the induced action of $\gal ( \mathbb{F}_{p^n}/ \mathbb{F}_p)$ on $S_r/ \mathbb{F}_{p^n}^{\times}$ is trivial. This proves the forward direction. 
\\
\indent \indent For the backwards direction, suppose $\mathbb{F}_{p^n}^\times$ trivializes $\gal (\mathbb{F}_{p^n}/\mathbb{F}_p)$ under the action of $\aut (R_{n,p})$ on $S_r$. Let $M \subseteq R_{n,p}$ be an extended submonoid. Since $I$ is maximal, $I^c \cap M$ is closed under multiplication with $\mathbb{F}_p^\times \subseteq I^c \cap M$. Thus $\hat{M} \doteq (I^c \cap M) \cup \left\{0 \right\} \subseteq M$ is an extended submonoid with $M \setminus \hat{M} \subseteq I$. We wish to show that $I$ covers $M$ with associated submonoid $\hat{M}$, for this it suffices to show that, if $G= I^c \cap M$, then any $\gamma \in (\Gamma_I^{\aut})^{\pi (G)}$ has a lift across $I$ lying in $(\Gamma_0^{\aut})^G$. 
\\
\indent \indent Note that $G$ is actually a subgroup of $R_{n,p}^\times$, since $R_{n,p}$ is finite and local. The decomposition $R_{n,p}^\times \cong \mathbb{F}_{p^n}^\times \times (I+1)$ is immediate. Moreover, any $1 \neq x \in I+1$ has order $p$ in $R_{n,p}^\times$, it follows that $R_{n,p}^\times \cong \mathbb{F}_{p^n}^\times \times (\mathbb{Z}/p\mathbb{Z})^n$. Since $\gcd (|\mathbb{F}_{p^n}^\times|,p)=1$, we may then write $G=H \times (\mathbb{Z}/p\mathbb{Z})^r$ for some $H \leq \mathbb{F}_{p^n}^\times$ and $0 \leq r \leq n$. Consider then the larger subgroup $\tilde{G}= \mathbb{F}_{p^n}^\times \times (\mathbb{Z}/p\mathbb{Z})^r \supseteq G$. We have $\tilde{G} \in S_r$, and by hypothesis the induced action of $\gal( \mathbb{F}_{p^n}/ \mathbb{F}_p)$ on $S_r/ \mathbb{F}_p^\times$ is trivial. Then $\theta (\sigma \tilde{G})= \theta (\tilde{G})$, so there is a $\lambda \in \mathbb{F}_{p^n}^\times$ such that $T_\lambda \sigma (\tilde{G})= \tilde{G}$. 
\\ 
\indent \indent Now we have an induced group automorphism $T_\lambda \sigma : \tilde{G} \to \tilde{G}$, this will necessarily factor as $T_\lambda \sigma = \psi \times \alpha$ for group automorphisms $\psi : \mathbb{F}_{p^n}^\times \to \mathbb{F}_{p^n}^\times$ and $\alpha: (\mathbb{Z}/p\mathbb{Z})^r \to (\mathbb{Z}/p\mathbb{Z})^r$. But $\mathbb{F}_{p^n}^\times$ is a cyclic group, so $H$ is actually a characteristic subgroup of $\mathbb{F}_{p^n}^\times$. In particular, $\psi (H)=H$, from which it follows that $T_\lambda \sigma (G)= (\psi \times \alpha) (H \times (\mathbb{Z}/p\mathbb{Z})^r)= H \times (\mathbb{Z}/p\mathbb{Z})^r=G$. Then we have that $T_\lambda \sigma \in (\Gamma_0^{\aut})^G$. 
\\
\indent \indent The desired conclusion is now a short step away. If $\phi$ is the Frobenius on $R_{n,p}/I \cong \mathbb{F}_{p^n}$, then the following commutes:
\begin{center}
\begin{tikzcd}
G \arrow[d, "\pi"] \arrow[rr, "T_\lambda \sigma"] &  & G \arrow[d, "\pi"] \\
\pi(G) \arrow[rr, "\phi"] &  & \pi(G)
\end{tikzcd}
\end{center}
so that $T_\lambda \sigma \in (\Gamma_0^{\aut})^G$ gives a lift along $I$ of $\phi$. It then follows that any $\gamma \in (\Gamma_I^{\aut})^{\pi (G)}= \aut (R_{n,p}/I)= \gal (\mathbb{F}_{p^n}/\mathbb{F}_p)$ has a lift along $I$, since $\phi$ generates $\gal (\mathbb{F}_{p^n}/\mathbb{F}_p)$.
\end{proof}
\end{lemma}
\indent \indent We remarked earlier that the class of finite local rings is a convenient setting for problems of this type. The justification for this is the fact that, for any extended submonoid $M$ in some finite local ring $R$ with maximal ideal $I$, the set $G= I^c \cap M$ forms a subgroup of $R^\times$. This was exploited in the proof of Lemma~\ref{lemma:2}.  
\\
\indent \indent The next lemma reduces the problem to a form in which we can solve it. Interestingly, the claim is roughly that each $\mathbb{F}_p$-subspace of $\mathbb{F}_{p^n}$ should behave like an ``eigenvector" for the Frobenius on $\mathbb{F}_{p^n}$, the corresponding ``eigenvalue" being multiplication by some element of $\mathbb{F}_{p^n}^\times$, considered as an $\mathbb{F}_p$-linear map. More precisely: 
\begin{lemma} \label{lemma:3}
Let $\phi$ be the Frobenius on $\mathbb{F}_{p^n}$. Then $\Gamma^{\aut}(R_{n,p})$ is a monoidal network if and only if for any $\mathbb{F}_p$-subspace $V \subseteq \mathbb{F}_{p^n}$, there is a $\lambda \in \mathbb{F}_{p^n}^\times$ such that $\phi (V)= \lambda V$. 
\begin{proof}
By Lemma~\ref{lemma:2}, we need only show that the right-hand side of the above holds if and only if $\mathbb{F}_{p^n}^\times$ trivializes $\gal (\mathbb{F}_{p^n}/\mathbb{F}_p)$ under the natural action of $\aut (R_{n,p})$ on $S_r$, for each $0 \leq r \leq n$. 
\\
\indent \indent For the forward direction, suppose $\Gamma^{\aut}(R_{n,p})$ is a monoidal network and let $V \subseteq \mathbb{F}_{p^n}$ be an $\mathbb{F}_p$-subspace. If we define: $H' \doteq \left\{ 1+\alpha \varepsilon: \alpha \in V \right\}$, then note that $H'$ is a subgroup of $R_{n,p}^\times$. Since any $1 \neq x \in H'$ has order $p$, we have $H \doteq \mathbb{F}_{p^n}^\times \times H' \in S_r$ for some $0 \leq r \leq n$. By hypothesis, $\mathbb{F}_{p^n}^\times$ trivializes $\gal (\mathbb{F}_{p^n}/\mathbb{F}_p)$ under the natural action of $\aut (R_{n,p})$ on $S_r$, it therefore follows that $T_\lambda \sigma$ stabilizes $H$ for some $\lambda \in \mathbb{F}_{p^n}^\times$. Then $T_\lambda \sigma$ must also stabilize $H'$, from which we get that $\phi (V)= \lambda^{-1} V$, as desired. 
\\
\indent \indent Now suppose that for any $\mathbb{F}_p$-subspace $V \subseteq \mathbb{F}_{p^n}$, there is a $\lambda \in \mathbb{F}_{p^n}^\times$ such that $\phi (V)= \lambda V$. Let $H \in S_r$ and define $V_H = \left\{ \alpha \in \mathbb{F}_{p^n}: 1+ \alpha \varepsilon \in H \right\}$. This is an additive subgroup of $\mathbb{F}_{p^n}$ and therefore an $\mathbb{F}_p$-subspace. By hypothesis, there is a $\lambda \in \mathbb{F}_{p^n}^\times$ with $\phi (V_H)= \lambda V_H$, or after relabeling, we may say that $\lambda \phi (V_H)=V_H$. Unraveling the definitions gives that $T_\lambda \sigma$ stabilizes $H$, and hence that $\theta (\sigma (H))= \theta (H)$. It follows that the induced action of $\gal (\mathbb{F}_{p^n}/ \mathbb{F}_p)$ on $S_r/ \mathbb{F}_p^\times$ is trivial, and thus that $\Gamma^{\aut}(R_{n,p})$ is a monoidal network by Lemma~\ref{lemma:2}.
\end{proof}
\end{lemma}
\indent \indent We are now adequately equipped to prove Theorem~\ref{thm:2}. We will prove the backwards direction first, as it is the easiest. The forward direction is handled for $n>4$ by a fairly straightforward counting argument, while the $n=4$ case is ruled out for most prime $p$ via a lower bound on Euler's $\varphi$ function found in~\cite{Rosser}. The remaining cases, of which there are two, are handled manually. 
\begin{proof}[Proof of Theorem 2]
We argue that the equivalent condition given in Lemma~\ref{lemma:3} occurs if and only if $n \leq 3$. Suppose first that $n \leq 3$. For $n=1$ the claim follows immediately. For $n=2$, simply note that on a one-dimensional $\mathbb{F}_p$-subspace $V= \spn \left\{\alpha \right\}$, we have $\phi (V)= \alpha^{p-1} V$. For $n=3$, we need only consider subspaces $V$ with $\dim V=2$. For this, it is enough to show that for any $\alpha \in \mathbb{F}_{p^3}$, there are $a,b,c,d \in \mathbb{F}_p$ with $\alpha^p= \frac{a\alpha+b}{c\alpha+d}$. Indeed, suppose this holds and let $V= \spn \left\{ \alpha^k, \alpha^r \right\}$, $k<r$, be a two-dimensional $\mathbb{F}_p$-subspace, where $\alpha$ is a generator of $\mathbb{F}_{p^3}^\times$. If the aforementioned claim holds, then $\alpha^{(r-k)p}= \frac{a\alpha^{r-k}+b}{c\alpha^{r-k}+d}$ for some $a,b,c,d \in \mathbb{F}_p$. This can be rearranged to give $\frac{\alpha^{rp}}{a\alpha^r+b\alpha^k}= \frac{\alpha^{kp}}{c\alpha^r+d\alpha^k} \doteq \lambda$. We then get $\phi (\alpha^k), \phi (\alpha^r) \in \lambda V$, thus $\phi (V)= \lambda V$. 
\\
\indent \indent Aiming to prove the original claim, let $\beta \in \mathbb{F}_{p^3}$. If $\beta \in \mathbb{F}_p$, taking $a=d=1$ and $b=c=0$ is sufficient. Otherwise, $\beta$ has degree 3 over $\mathbb{F}_p$, thus $\left\{1, \beta, \beta^2 \right\}$ is an $\mathbb{F}_p$-basis of $\mathbb{F}_{p^3}$. Suppose $\beta^p= a' \beta^2+b' \beta+c'$. If $a'=0$ we are done, so suppose $a' \neq 0$ and let $x^3+mx^2+nx+q \in \mathbb{F}_p[x]$ be the minimal polynomial of $\beta$ over $\mathbb{F}_p$. Defining $c= \frac{1}{a'}$, $d= \frac{m}{a'}- \frac{b'}{(a')^2}$, $a= cc'+db'-n$ and $b= dc'-q$, the equation $\beta^p = \frac{a\beta +b}{c\beta+d}$ is easy to verify. This proves the backwards direction. 
\\
\indent \indent Now suppose that for any $\mathbb{F}_p$-subspace $V \subseteq \mathbb{F}_{p^n}$, $\phi (V)= \lambda V$ for some $\lambda \in \mathbb{F}_{p^n}^\times$. Letting $\alpha \in \mathbb{F}_{p^n}$, consider the subspace spanned by $\left\{1, \alpha \right\}$. By hypothesis, there are $a,b,c,d \in \mathbb{F}_p$ and $\lambda \in \mathbb{F}_{p^n}^\times$ such that $1= \lambda (a+b\alpha)$ and $\alpha^p=\lambda (c+d\alpha)$. In particular, every $\alpha \in \mathbb{F}_{p^n}$ satisfies $\alpha^p= \frac{a\alpha+b}{c\alpha+d}$ for some $a,b,c,d \in \mathbb{F}_p$. 
\\
\indent \indent The case $p=2$ requires special treatment, we handle this now. In this case, choose some $\alpha \in \mathbb{F}_{2^n}$ of degree $n$ over $\mathbb{F}_2$. The equation $\alpha^2= \frac{a+b\alpha}{c+d\alpha}$ gives a polynomial over $\mathbb{F}_2$ of degree at most 3 satisfied by $\alpha$, it follows that $n \leq 3$ as desired. In the sequel we then assume $p>2$. 
\\
\indent \indent We now endeavor to show that $n \leq 4$. On the $\mathbb{F}_p$-vector space $\mathbb{F}_p^4$, we have a projection $\mathbb{F}_p^4 \to \mathbb{F}_p^2$ given by $(v_1, v_2, v_3, v_4) \mapsto (v_3, v_4)$, we let $U$ denote the kernel of this map. We additionally let $\mathbb{P}(\mathbb{F}_p^4)$ denote the projectivization of $\mathbb{F}_p^4$, and $W$ the image of $\mathbb{F}_p^4 \setminus U$ under the canonical map $\pi: \mathbb{F}_p^4 \setminus \left\{0 \right\} \to \mathbb{P}(\mathbb{F}_p^4)$. Now for each $\alpha \in \mathbb{F}_{p^n}$, we let $S_{\alpha}$ be the set of all $(a,b,c,d) \in \mathbb{F}_p^4$ with the property that $\alpha^p= \frac{a\alpha +b}{c\alpha+d}$. We previously showed that $S_\alpha \neq \emptyset$. Let $\hat{T}: \left\{ S_\alpha \right\}_{\alpha \in \mathbb{F}_{p^n}} \to \mathbb{F}_p^4$ have $\hat{T}(S_\alpha) \in S_\alpha$. This yields a (set-theoretic) map $T: \mathbb{F}_{p^n} \to W$ given by $T: \alpha \mapsto \pi (\hat{T}(S_\alpha))$, which in turn descends to an injection $\mathbb{F}_{p^n}/\!{\sim} \lhook\joinrel\longrightarrow W$, where $\sim$ is the equivalence relation induced by $T$. We thus obtain $\left|\mathbb{F}_{p^n}/\!{\sim} \right| \leq |W|= \frac{p^4-p^2}{p-1}$. 
\\
\indent \indent Now fix some $\alpha \in \mathbb{F}_{p^n}$. If $\beta \sim \alpha$ and $\hat{T}(S_\alpha)= (a,b,c,d)$, then $\hat{T}(S_\beta)= \lambda (a,b,c,d)$ for some $\lambda \in \mathbb{F}_p^\times$. We deduce that any $\beta \sim \alpha$ must be a root of the polynomial $x^p(cx+d)-(ax+b)$, and therefore the equivalence class of any $\alpha$ under $\sim$ contains at most $p+1$ elements. Denoting a general equivalence class as $[\alpha]$ and choosing a representative $x_i$ for each class, we have that:
\begin{align*}
p^n= \displaystyle \sum_i |[x_i]| \leq (p+1) \left|\mathbb{F}_{p^n}/\!{\sim} \right|
\end{align*}
which implies that $\left|\mathbb{F}_{p^n}/\!{\sim} \right| \geq \frac{p^n}{p+1}$. This, together with the previously obtained inequality, gives $\frac{p^n}{p+1} \leq \frac{p^4-p^2}{p-1}$, or $p^n \leq (p(p+1))^2<p^5$ for $p>2$. This proves that $n \leq 4$. 
\\
\indent \indent Restricting our attention now to the $n=4$ case, $p^2+1$ is a divisor of $p^4-1$, there are therefore exactly $\varphi (p^2+1)$ elements of order $p^2+1$ in $\mathbb{F}_{p^4}^\times$. Let $\alpha$ have order $p^2+1$ and let $\alpha^p= \frac{a \alpha +b}{c\alpha+d}$. We will narrow the possible values of the coefficients here and subsequently use this to bound $\varphi (p^2+1)$ from above. 
\\
\indent \indent We first note that $c\neq 0$. Indeed, suppose to the contrary, and after relabeling, we may then write $\alpha^p= a
\alpha +b$. Then $\frac{1}{\alpha}= \alpha^{p^2}=a\alpha^p+b$, implying $a\alpha^{p+1}+b\alpha-1=0$. We also have that $\alpha^{p+1}-a\alpha^2-b\alpha=0$. From these two, we obtain: $a^2\alpha^2+(b+ab)\alpha-1=0$. Now, $a\neq 0$, otherwise $\alpha \in \mathbb{F}_p$, so this implies that $\alpha$ has degree at most 2 over $\mathbb{F}_p$. This is impossible, any element of $\mathbb{F}_{p^4}$ of multiplicative order $p^2+1$ cannot lie in any proper subextension of $\mathbb{F}_{p^4}$. We conclude that $c \neq 0$. Because of this, we may scale to obtain, after relabeling: $\alpha^p = \frac{a\alpha+b}{\alpha +d}$. 
\\
\indent \indent A similar argument, applied to $\alpha^p = \frac{a\alpha+b}{\alpha +d}$ gives $a\alpha^{p+1}-\alpha^p+b\alpha -d=0$, as well as $\alpha^{p+1}+d\alpha^p -a\alpha -b=0$. From this, we obtain: $(-1-ad)\alpha^p+(a^2+b)\alpha +(ab-d)=0$. Now, we previously showed that $\alpha^p \neq k \alpha +r$ for any $k,r \in \mathbb{F}_p$. Then the following must each hold:
\begin{align*}
ad+1 &=0\\
a^2+b &=0 \\
ab-d &=0.
\end{align*}
Evidently from this, $a,d \neq 0$, and therefore $b \neq 0$ as well. We then obtain $a=- \frac{1}{d}, b= - \frac{1}{d^2}$ and $d^4=1$. For each fourth-root of unity $d \in \mathbb{F}_p$, let $p_d (x) = (x+d)x^p + \frac{1}{d}x+ \frac{1}{d^2}$ and let $Z (p_d)$ denote the zero set of $p_d$ in some algebraic closure of $\mathbb{F}_p$. The above says that the collection of elements in $\mathbb{F}_{p^4}^\times$ of order $p^2+1$ is contained in $\displaystyle \bigcup_{d^4=1} Z(p_d)$. We then have $\varphi (p^2+1) \leq 4(p+1)$ if $p \not\equiv -1 \pmod{4}$ and $\varphi (p^2+1) \leq 2 (p+1)$ if $p \equiv -1 \pmod{4}$, by Euler's Criterion. 
\\
\indent \indent Courageously pressing forward, we have the lower bound~\cite{Rosser}:
\begin{align*}
\varphi (p^2+1) > \frac{p^2+1}{e^\gamma \log \log (p^2+1)+ \frac{3}{\log \log (p^2+1)}}
\end{align*}
where $\gamma$ is the Euler-Mascheroni constant. This, together with $\varphi (p^2+1) \leq 4(p+1)$, gives:
\begin{align*}
\frac{4(p+1)}{p^2+1}(e^\gamma \log \log (p^2+1)+ \frac{3}{\log \log (p^2+1)}) \geq 1.
\end{align*}
However, the function $f(x)= \frac{4(x+1)}{x^2+1}(e^\gamma \log \log (x^2+1)+ \frac{3}{\log \log (x^2+1)})$ decreases monotonically on $[3, \infty)$ and we have $f(x)<1$ for $x \geq 23$. We deduce that $p <23$. 
\\
\indent \indent What remains is to rule out the finite set of remaining cases. The only primes $3 \leq p <23$ satisfying $\varphi (p^2+1) \leq 4(p+1)$ are $p=3,5$ or 7. For $p=7 \equiv -1 \pmod{4}$, we actually need $\varphi (p^2+1) \leq 2(p+1)$, which fails for $p=7$, thus the only cases to be checked are $p=3$ or $5$. For $p=3$, the polynomial $x^4+x^2-1 \in \mathbb{F}_3[x]$ is irreducible over $\mathbb{F}_3$, it therefore has a root in $\mathbb{F}_{3^4}$, say $\beta$. If we had $\beta^3= \frac{a\beta +b}{c\beta +d}$ for some $a,b,c,d \in \mathbb{F}_3$, then $\beta$ would be a root of the polynomial $cx^4+dx^3-ax-b$, implying this is an $\mathbb{F}_3$-multiple of $x^4+x^2-1$, a contradiction. A similar argument applied to the irreducible polynomial $x^4+x^2+2 \in \mathbb{F}_5[x]$ gives the result for $p=5$. This completes the proof.
\end{proof}
\indent \indent We conclude with a non-exhaustive discussion of potential directions in which this could be taken. The careful reader has likely noticed by now that the general existence question has not been touched: what properties must $R$ have to guarantee that $R$ admits a monoidal network? We have given examples in both directions, e.g. fields do not admit monoidal networks for trivial reasons, whereas examples of rings which do are scattered throughout the preceding text. More generally, we may consider the collection\footnote{We consider only subgroups of $\sym (R)$ here purely for the purpose of avoiding set-theoretic difficulties. Ideally, one would want to form a category of monoidal networks, then consider top action groups of isomorphism classes of monoidal networks here. Depending on how the morphisms are defined, this may still be too big to be a set, but generally one may consider only classes of monoidal networks for which the top action group acts faithfully on $R$ to avoid this inconvenience. All of this can be done, but in the interest of brevity we omit this formality.}: 
\begin{align*}
\mathcal{M}(R) \doteq \left\{ G \leq \sym (R) : (G, \iota) = (\Gamma_0, \gamma_0) \textrm{ for some monoidal network } \Gamma (R) \right\}
\end{align*}
where $\iota : G \hookrightarrow \sym (R)$ is the inclusion. The existence question then asks when $\mathcal{M}(R) \neq \emptyset$, but we can ask much more if we first examine $\mathcal{M}(R)$ for $R$ a finite local ring. Partially ordering $\mathcal{M}(R)$ by inclusion, we have the following result:
\begin{prop} \label{prop:9}
Let $R$ be a finite local ring that is not a field. Then $\mathcal{M}(R)$ is upward-closed and nonempty. In particular, $R$ admits a monoidal network.
\begin{proof}
For the first claim, let $G \in \mathcal{M}(R)$ and let $G \leq H \leq \sym (R)$. Then $G$ is the top action group of some monoidal network defined over $R$, say $\Gamma (R)$. Define a new multiaction $\Gamma' (R)$ by setting, for $I \neq 0$, $\Gamma_I' \doteq \Gamma_I$, acting on $R/I$ in the manner dictated by $\Gamma (R)$. Let $\Gamma_0= H$, the action on $R$ being induced by the inclusion $H \hookrightarrow \sym (R)$. It follows immediately that $\Gamma' (R)$ is a monoidal network on $R$ since $\Gamma (R)$ is, so $H \in \mathcal{M}(R)$. 
\\
\indent \indent For the second part, we will show that $\sym (R) \in \mathcal{M}(R)$. For $0 \neq I \subseteq R$, assign $\Gamma_I = \aut (R/I)$, and take $\Gamma_0 = \sym (R)$. With each of these acting naturally on their respective quotients, this forms a multiaction on $R$, call this $\Gamma (R)$. Let $I$ be the unique maximal ideal contained in $R$ with residue field $R/I \cong \mathbb{F}_{p^n}$, and let $\phi$ be the Frobenius on $R/I$. Fix an extended submonoid $M$ in $R$, we argue that $I$ is a covering ideal of $M$ with $\hat{M}=M$ under this multiaction. For $x \in M \setminus I$, let $[x]$ denote the equivalence class of $x$ in $R/I$. Then it follows from the conditions on $R$ that $x^{-1}$ is a power of $x$, and moreover that $[x] \cap M = x \cdot ( M \cap (1+I))$. In particular, for any $a,b \in R/I$, $a,b \neq 0$ with $\pi^{-1}(a) \cap M, \pi^{-1}(b) \cap M \neq \emptyset$, we can find a bijection $\gamma_{a,b} : \pi^{-1}(a) \cap M \to  \pi^{-1}(b) \cap M$. We then define $\tilde{\phi}: R \to R$ by:
\begin{align*}
\tilde{\phi}: x \mapsto \begin{cases} 
      x & x \notin M, x \in I \cap M \\
      \gamma_{a, \phi (a)} (x) & x \in \pi^{-1}(a) \cap M, a \neq 0.  
   \end{cases}
\end{align*}
It is easy to see that $\tilde{\phi} \in \Gamma_0^M$. Moreover, by construction $\tilde{\phi}$ lifts $\phi$ along $I$ and since $\phi$ generates $\aut (R/I)$, $I$ is a (proper, nonzero) covering ideal of $M$. Then $\Gamma (R)$ forms a monoidal network on $R$ with $\Gamma_0= \sym (R)$.
\end{proof}
\end{prop}
\noindent In view of the above, Theorem~\ref{thm:2} says that for $R=R_{n,p}$, $\mathcal{M}(R_{n,p})$ is a filter on the partially ordered set (poset) of subgroups of $\sym (R)$ when $n \leq 3$. In fact, Theorem~\ref{thm:2} is equivalent to the statement that $\mathcal{M}(R_{n,p})$ is the principal filter generated by $\aut (R_{n,p})$ if and only if $n \leq 3$. Additionally considering Proposition~\ref{prop:9}, it is natural then to inquire about when exactly $\mathcal{M}(R)$ has the structure of a filter on the subgroup poset of $\sym (R)$. For a finite local ring $R$ that is not a field, this is equivalent to asking when there is a unique minimal subgroup $G \leq \sym (R)$ occurring as the top action group of some monoidal network on $R$ (note that the potential existence of non-principal filters on the subgroup poset of $\sym (R)$ for infinite $R$ prevents us from asserting this equivalence in the infinite case). 
\\
\indent \indent It is easy to verify that the first part of Proposition~\ref{prop:9} holds without the assumption that $R$ is finite and local, $\mathcal{M}(R)$ is upward-closed for any $R$. In contrast, $\mathcal{M}(R)$ can be either empty or nonempty if the finiteness condition on $R$ is dropped. For instance, $\mathbb{R}[x]/(x^2)$ is a local ring which admits a monoidal network: the automorphic multiaction on this ring is a monoidal network because $\mathbb{R}$ is rigid. Less trivially, $\mathcal{M}( \overline{\mathbb{F}_{p}}[x]/(x^2)) \neq \emptyset$, where $\overline{\mathbb{F}_{p}}$ is the algebraic closure of $\mathbb{F}_p$. This follows by replicating the argument from Proposition~\ref{prop:9}, noting that for this ring, the equality $[x] \cap M = x \cdot ( M \cap (1+I))$ still holds for any extended submonoid $M$ and $x \in M \setminus I$, since $x^{-1}$ is still a power of $x$ here. We also have a class of examples in the other direction:
\begin{prop} \label{prop:10}
Let $F$ be an uncountable algebraically closed field. Then $F[x]/ (x^2)$ admits no monoidal network.
\begin{proof}
We let $R \doteq F[x]/ (x^2)$ for brevity. Suppose to the contrary that $\Gamma (R)$ is some monoidal network defined on $R$. Since $F$ is uncountable, $F$ has infinite transcendence degree over the prime subfield $k \subseteq F$. Pick a transcendence basis $S$ for $F$ over $k$ and let $\alpha, \beta \in S$ be distinct. Let $M$ be the multiplicative submonoid of $R$ generated by $k$, $\alpha$ and all elements of the form $\beta +r \varepsilon$ with $r \in F$, where as before $\varepsilon \in R$ is the image of $x \in F[x]$ under projection. Explicitly: $M= \left\{ a \alpha^m (\beta+r\varepsilon )^n : a \in k, m,n \geq 0, r \in F \right\}$. 
\\
\indent \indent It is not immediately clear that $M$ is an extended submonoid of $R$, we prove this now. As in the proof of Theorem~\ref{thm:2}, we identify the subgroup of $\aut (R)$ generated by the $T_\lambda :a+b \varepsilon \mapsto a+b\lambda \varepsilon$ with $F^\times$ and the subgroup generated by the maps of the form $a+b\varepsilon \mapsto \phi (a) +\phi (b) \epsilon$, for $\phi \in \aut (F)$, with $\aut (F)$. Then $\aut (R)= F^\times \rtimes \aut (F)$ and the fixed subring of $R$ is $F^{\aut (F)} \subseteq R$. Since $F$ is algebraically closed, $F^{\aut (F)}=k$, thus $k \subseteq M$ implies that $M$ is an extended submonoid of $R$ as claimed. Since $\Gamma (R)$ is a monoidal network, the unique proper, nonzero ideal $I= (\varepsilon )$ in $R$ must be a covering ideal of $M$. For the associated submonoid we must have $\hat{M}=M$, since any nonzero $x \in M$ is a unit in $R$. 
\\
\indent \indent Now we have $\pi (M) = \left\{ a \alpha^m \beta^n: a \in k, m,n \geq 0 \right\}$. Furthermore, note that the transposition $(\alpha \beta)$ on the transcendence basis $S$ induces an automorphism $\gamma :F \to F$ with $\gamma (\pi (M))= \pi (M)$. The action of $\Gamma_I$ on $R/I \cong F$ extends that of $\aut (F)$, there is thus some $g \in \Gamma_I^{\pi (M)}$ with $g \cdot x= \gamma (x)$ for each $x \in F$. Note then that the lift $\tilde{g}: M \to M$ of $g$ across $I$ induces injections:
\begin{align*}
\tilde{g}: \pi^{-1}(\alpha) \cap M \hookrightarrow \pi^{-1} (\beta) \cap M \\
\tilde{g}: \pi^{-1}(\beta) \cap M \hookrightarrow \pi^{-1} (\alpha) \cap M
\end{align*}
so that $|\pi^{-1}(\alpha) \cap M|= |\pi^{-1}(\beta) \cap M|$. However, by construction $|\pi^{-1}(\beta) \cap M| \geq |F|$, while the algebraic independence of $\alpha$ and $\beta$ over $k$ implies $|\pi^{-1}(\alpha) \cap M|=1$, a contradiction.
\end{proof}
\end{prop}
\indent \indent The final question we wish to consider concerns Proposition~\ref{prop:10}. The proof given above relies heavily on the existence of a transcendence basis for $F$ over $k$, this comes into play not only in the construction of $M$, where conceivably it could be avoided, but also implicitly in the claim that the subfield of $F$ fixed by $\aut (F)$ is $k$. It seems quite difficult to carry the proof of Proposition~\ref{prop:10} forward without the knowledge afforded to us by the axiom of choice here. In its absence, little can actually be said about $F^{\aut (F)}$ and therefore about the structure of extended submonoids in $F[x]/(x^2)$ as well. We are then placed at a fundamental disadvantage for the study of monoidal networks on $F[x]/(x^2)$. Consequently, we pose the following question:
\begin{ques}
Is Proposition~\ref{prop:10} a theorem of ZF?
\end{ques}
\noindent It is the view of the author that an answer in either direction to the question above would be interesting. If the answer is in the affirmative, then this would prove that the existence of uncountable, rigid algebraically closed fields is inconsistent with ZF, since for rigid $F$, the automorphic multiaction on $F[x]/(x^2)$ is a trivial monoidal network. On the other hand, if the answer is negative, then this would demonstrate that some (perhaps weaker) form of the axiom of choice is required to prove Proposition~\ref{prop:10}. 
\\
\indent \indent It seems that the answer to the question may be ``no," and we close with some support for this viewpoint. For this, we examine the occurrence of a weaker property of multiactions on $R_F \doteq F[x]/ (x^2)$ for uncountable algebraically closed fields $F$. Fix an infinite, well-ordered cardinal $\aleph_{\alpha}$ and a multiaction $\Gamma (R_F)$ on $R_F$. We say that $\Gamma (R_F)$ is an $\aleph_{\alpha}$-\emph{monoidal network} if any extended submonoid $M \subseteq R_F$ satisfying $|\pi (M) \setminus F^{\aut (F)}| \leq \aleph_{\alpha}$ has a proper, nonzero covering ideal under $\Gamma (R_F)$. Note that the argument for Proposition~\ref{prop:10} actually proves that, for any uncountable algebraically closed field $F$ and infinite $\aleph_{\alpha}$, $R_F$ admits no $\aleph_{\alpha}$-monoidal network. Indeed, the extended submonoid $M$ constructed in the proof of Proposition~\ref{prop:10} has $|\pi (M)|= \aleph_0$ (assuming choice) and no proper, nonzero covering ideal under any multiaction $\Gamma (R_F)$. Thus, in the presence of the axiom of choice, we have:
\begin{prop} \label{prop:11}
For any uncountable algebraically closed field $F$ and infinite $\aleph_{\alpha}$, $R_F$ admits no $\aleph_{\alpha}$-monoidal network.
\end{prop}
\noindent \noindent This claim cannot be proven without the axiom of choice, in fact, there is a model of ZF$+$DC in which $R_{\mathbb{C}}$ admits an $\aleph_0$-monoidal network, so even dependent choice is not strong enough to prove this. 
\begin{prop} \label{prop:12}
In any model of ZF$+$DC in which all sets in $\mathbb{R}$ have the property of Baire, $R_{\mathbb{C}}$ admits an $\aleph_0$-monoidal network.
\begin{proof}
Under such circumstances, as is noted in \cite{Rosendal}, any homomorphism of additive abelian groups $\mathbb{C} \to \mathbb{C}$ is necessarily continuous and it therefore follows that the only field automorphisms of $\mathbb{C}$ are the identity and conjugation. In particular, the field fixed by the automorphism group of $\mathbb{C}$ is $\mathbb{R}$ in this case, and the extended submonoids of $R_{\mathbb{C}}$ are precisely the multiplicative submonoids $M$ with $\mathbb{R} \subseteq M$. The claim essentially follows from the lemma below: 
\begin{lemma} \label{lemma:4}
Let $M \subseteq R_{\mathbb{C}}$ be multiplicatively closed with $\mathbb{R} \subseteq M$. Let $\pi: R_{\mathbb{C}} \to \mathbb{C}$ be the projection, and suppose that $\pi^{-1} (z) \cap M$ and $\pi^{-1} ( \overline{z}) \cap M$ are both nonempty for some fixed $z$. Then $|\pi^{-1} (z) \cap M| = |\pi^{-1} ( \overline{z}) \cap M|$. 
\begin{proof}
The claim is obvious when $z=0$, so suppose $z \neq 0$. Take any $z+ \alpha \varepsilon, \overline{z}+ \beta \varepsilon \in M$. Then we have $\frac{1}{|z|^2} (z+\alpha \varepsilon)^2 (\overline{z}+ \beta \varepsilon) \in M$. However, actually computing this gives:
 \begin{align*}
 \frac{1}{|z|^2} (z+\alpha \varepsilon)^2 (\overline{z}+ \beta \varepsilon) = z+ \frac{1}{\overline{z}}(\beta z+ 2\alpha \overline{z}) \varepsilon \in \pi^{-1} (z) \cap M.
\end{align*}  Thus, fixing some $z+ \alpha_0 \varepsilon \in M$, we obtain an injection $\pi^{-1}(\overline{z}) \cap M \hookrightarrow \pi^{-1}(z) \cap M$ given by:
\begin{align*}
\overline{z}+ \beta \varepsilon \mapsto \frac{1}{|z|^2} (z+\alpha_0 \varepsilon)^2 (\overline{z}+\beta \varepsilon).
\end{align*} 
The lemma follows by interchanging the roles of $z$ and $\overline{z}$ and applying the Cantor-Schroeder-Bernstein theorem.
\end{proof}
\end{lemma}
\indent \indent In view of Lemma~\ref{lemma:4}, define a multiaction $\Gamma( R_{\mathbb{C}})$ on $R_{\mathbb{C}}$ by setting $\Gamma_I= \aut (R_{\mathbb{C}}/I)$ for $I \neq 0$ and $\Gamma_0(R_{\mathbb{C}})= \sym (R_{\mathbb{C}})$. Let $M \subseteq R_{\mathbb{C}}$ be an extended submonoid with $\pi (M) \setminus \mathbb{R}$ countable. Letting $I= (\varepsilon )$ henceforth, note that $\Gamma_I^{\pi (M)}$ is either trivial or generated by $\phi :z \mapsto \overline{z}$. If the former holds, then $I$ is a covering ideal of $M$ with $\hat{M}=M$, so suppose the latter holds. Then $\phi$ stabilizes $\pi (M)$, so for any $z \in \mathbb{C}$, either $\pi^{-1} (z) \cap M, \pi^{-1}(\overline{z}) \cap M$ are both empty or both nonempty. Setting $\mathcal{A}= \left\{ z \in \pi (M) \setminus \mathbb{R}:  \pi^{-1} (z) \cap M, \pi^{-1}(\overline{z}) \cap M \neq \emptyset \right\}$, notice that $\mathcal{A}$ is countable by hypothesis, and Lemma~\ref{lemma:4} implies the existence of a bijection $\gamma_z : \pi^{-1} (z) \cap M \to  \pi^{-1}(\overline{z}) \cap M$ for any $z \in \mathcal{A}$ with $\Im (z)>0$. We may now assemble these bijections into a global bijection, since there are only countably many of them. Define $\tilde{\phi}: R_{\mathbb{C}} \to R_{\mathbb{C}}$ by:
\begin{align*}
\tilde{\phi}: x \mapsto \begin{cases} 
      x & x \notin M \textrm{ or } \pi (x) \in \mathbb{R} \\
      \gamma_{\pi (x)} (x) & x \in M, \Im (\pi (x))>0 \\
      \gamma_{\overline{\pi (x)}}^{-1}(x) & x \in M, \Im (\pi (x))<0. 
   \end{cases}
\end{align*}
We have $\tilde{\phi} \in \Gamma_0^{M}$ and $\pi \tilde{\phi}=\phi \pi$ on $M$, thus $\tilde{\phi}$ gives a lift of $\phi$ along $I$ and so $I$ is a covering ideal of $M$ with $\hat{M}=M$. Then $\Gamma (R_{\mathbb{C}})$ is an $\aleph_0$-monoidal network on $R_{\mathbb{C}}$.
\end{proof}
\end{prop}
\noindent Our original claim that such a model exists now follows from the constructions of Shelah in~\cite{Shelah}, where he exhibits a model of ZF$+$DC in which all subsets of $\mathbb{R}$ have the property of Baire. Notice that the $\aleph_0$-monoidal network constructed in Proposition~\ref{prop:12} may very well be a full monoidal network, though we have not proven this. The obstruction is purely set-theoretic: we cannot necessarily assemble the bijections $\gamma_z$ as we did above if the set $\mathcal{A}$ is uncountable, at least not under the set-theoretic hypotheses that apply here. Finding a way to circumvent this difficulty would prove that Proposition~\ref{prop:10} is not a theorem of ZF. 
\\
\\
\noindent \textbf{\emph{Acknowledgments.}} The author would like to thank Ashwin Trisal for many constructive comments and discussions related to this topic. 

\end{document}